\documentclass[11pt,a4paper]{article}

\usepackage{amssymb}
\usepackage{amsmath}
\usepackage{amsthm}
\usepackage[numbers]{natbib}
\usepackage{microtype}
\usepackage[final]{showkeys}
\usepackage{graphicx}
\usepackage{bm}
\usepackage{color}
\usepackage{mathbbol}
\usepackage[utf8]{inputenc}
\usepackage{subfig}
\usepackage{cmbright}
\usepackage[page]{appendix}
\usepackage[font={small,it}]{caption}

\newtheorem{Theorem}{Theorem}[section]
\newtheorem{Proposition}{Proposition}[section]
\newtheorem{Remark}{Remark}[section]
\newtheorem{Definition}{Definition}[section]

\usepackage[hmargin=3cm,vmargin={3cm,3.5cm}]{geometry}
\numberwithin{equation}{section}
\allowdisplaybreaks

\title{Input-output consistency in perfect integrate and fire interconnected neurons}

\begin{document}

	\author{$\text{Petr Lansky}_1\!\!\footnote{We started this work several years ago. Unfortunately, we had substantial delay and last year our colleague and friend Petr Lansky passed out. His contribution to the results was fundamental and we dedicate the outcome of this effort to his memory.}$\vspace{.5cm} , $\text{Federico Polito}_2$, $\text{Laura Sacerdote}_2$\\
		\footnotesize (1) -- Institute of Physiology, Academy of Sciences of the Czech Republic\\ \footnotesize Videnska 1083, Praha 4, Czech Republic\\
		\footnotesize (2) -- Dipartimento di Matematica, Universit\`a di Torino\\ \footnotesize Via Carlo Alberto 10, 10123, Torino, Italy
}
	
	\maketitle

	\begin{abstract} 
		\noindent Interspike intervals describe the output of neurons.
		Signal transmission in a neuronal network implies that the output of some neurons becomes the input of others.
        The output should reproduce the main features of the input to avoid a distortion when it becomes the input of other neurons, that is input and output should exhibit some sort of consistency.
        In this paper, we consider the question: how should we mathematically characterize the input in order to get a consistent output?
        Here we interpret the consistency by requiring the reproducibility of the input tail behaviour of the interspike intervals distributions in the output.
        Our answer refers to a system of interconnected neurons with stochastic perfect integrate and fire units.
        In particular, we show that the class of regularly-varying vectors is a possible choice to obtain such consistency. Some further necessary technical hypotheses are added.
	\end{abstract}
	
    \medskip

	
	\section{Introduction}

        A key aspect in neuroscience studies regards signal transmission in neuronal systems.
        The use of mathematical models for neuronal activity is a natural tool to perform such studies.
        Clearly, to investigate the signal transfer, the input and output part of the
        model has to be specified. There is an abundant number of studies where the authors
        have made a set of assumptions about the input properties and the output
        of the target neuron is investigated. The usual assumptions contain the terms like
        Poissonian noise and independence. By releasing these assumptions, more and more
        complex models have been considered. However, up to our knowledge, the consistency
        between the input  properties and neuronal output has never been investigated.
        The straightforward question posed in this paper is how to modify the mathematical properties of the
        input to achieve the same properties of the output neurons. As seen below, the
        simplicity of the question does not imply a simple answer.

        Biophysical and more abstract, mathematical, models of   neurons aim to describe
        signal transmission within a neuronal network. The    models are intrinsically stochastic
        for at least two reasons. At first, the experimentally observed neuronal activity is up
        to a certain extent always characterized by random fluctuations. This randomness is
        apparent in the highly irregular activity of neurons despite keeping the experimental
        conditions fixed. The lack of reproducibility of the individual firing patterns may be
        seen as an argument for the rate coding hypothesis (early proposed by \cite{adrian1928basis}; see also \cite{gestner2002spiking,tuckwell1988introduction}) and the deterministic neuronal models would be a sufficient description of the reality. However, the rate code would
        be rather inefficient and hardly realistic \cite{Brette2015}. Therefore, the rate code has to be replaced or
        at least extended and the alternatives known under different names -- variability coding,
        correlation coding, temporal coding -- investigated. From this fact stems out the second
        reason for stochastic modeling. It serves as a tool to quantify the information contained
        in the neuronal activity. 
        
        Stochastic Integrate and Fire (IF) models describe the interspike intervals (ISIs) as first passage times of a stochastic process through a boundary and own their success to their capability of conjugating some sort of biological realism with a reasonable mathematical tractability
        \cite{Burkitt20061,Burkitt200697,gerstner2014neuronal,Sacerdote201399,tuckwell1988introduction}.
        Since the seminal paper \cite{gerstein1964random}, many modifications of such model were proposed with the aim to improve the realism without losing the actual tractability \cite{Clusella2022,D'Onofrio2018,Jahn2011563,Jaras20215249}. 
        Generally, these stochastic IF models are Markovian thanks to the assumptions about the input from the surrounding neurons that is hypothesized Poissonian \cite{Kass20058,ricciardi} and to the independence assumption on superimposed inputs \cite{Abeles1982317,Ascione20206481,ricciardi,Tamborrino201445} (see \cite{ASCIONE20201130} for an example of semi-Markov IF model). However, as discussed in \cite{Lindner2006} the superposition of Poisson inputs does not determine a Poisson process if the inputs are not independent (see also \cite{Averbeck2009310}). Temporal correlations of the input may determine different non-Poissonian distributions of the global input whose effects survive in the output \cite{aaa}. Moreover, \cite{rosenbaum2010pooling} shows that strong correlations in postsynaptic membrane potentials may arise due to correlations in the input. As far as the output is concerned, often
        experimental data show ISIs distributions with heavy tails \cite{Choudhary2020148,Kass2018183,Kim200539,Teka2017110}. However to obtain such distributions from stochastic IF models one has to force the dynamics introducing specific boundaries \cite{Chacron2003253,Chacron2007301,Persi20042577}, making use of random time-changes \cite{Ascione2019}, or going back to the special case of no drift in the original Gerstein and Mandelbrot model. 

        Signal transmission implies that the output of each neuron becomes the input of the next ones. In order to model the membrane potential dynamics of neurons through a stochastic IF, it becomes natural to investigate under which hypotheses there is consistence between input and output. That is, to recognize a class of models for which the output reproduces the features requested for the input. This paper is a first step in this direction: we define a set of biologically meaningful hypotheses for the input and we prove that these features are retained by the output of a simple discrete time and discrete state space IF model. The main feature we aim to reproduce is the existence of heavy tails for the ISIs distribution, to which we add some technical hypotheses that are compatible with the biological framework.
        Gerstein and Mandelbrot \cite{gerstein1964random} considered stable distributed random variables as output of their IF model to obtain a good fit to data exhibiting heavy tails. However, input characterized by stable distributions do not reproduce the stability property in the output.
        Hence, we enlarge the class of random variables used to model ISIs and propose the use of regularly-varying vectors. We model the ISIs of $n$ neurons as regularly-varying vectors and we prove that the IF paradigm preserves this feature, if we add some technical hypotheses.
        
        The considered model is presented in Section \ref{descr}, while in Section \ref{IOC} we prove that input-output consistency is achieved. Finally, Section \ref{conc} reports some conclusions.

	\section{Model description}\label{descr}
	
	    
	   The object of our study will be a set of dependent stochastic
        point processes, producing an input signal, and a set of perfect integrators fed by those processes. The presence of dependence in the input processes is motivated by the hypothesized existence of an underlying network structure in which communication between vertices is controlled by a stochastic IF mechanism. In the following, we will not be interested in the topological structure of the underlying network, but rather we fix the form of the dependence in the input signal and we proceed the analysis as it is detailed below.
        In particular, for our aims it is not necessary to detail the structure of the connections while we fix the wished stochastic dependency determined by them.
        
%
	    \begin{figure}
	        \centering
            \includegraphics[scale=.35]{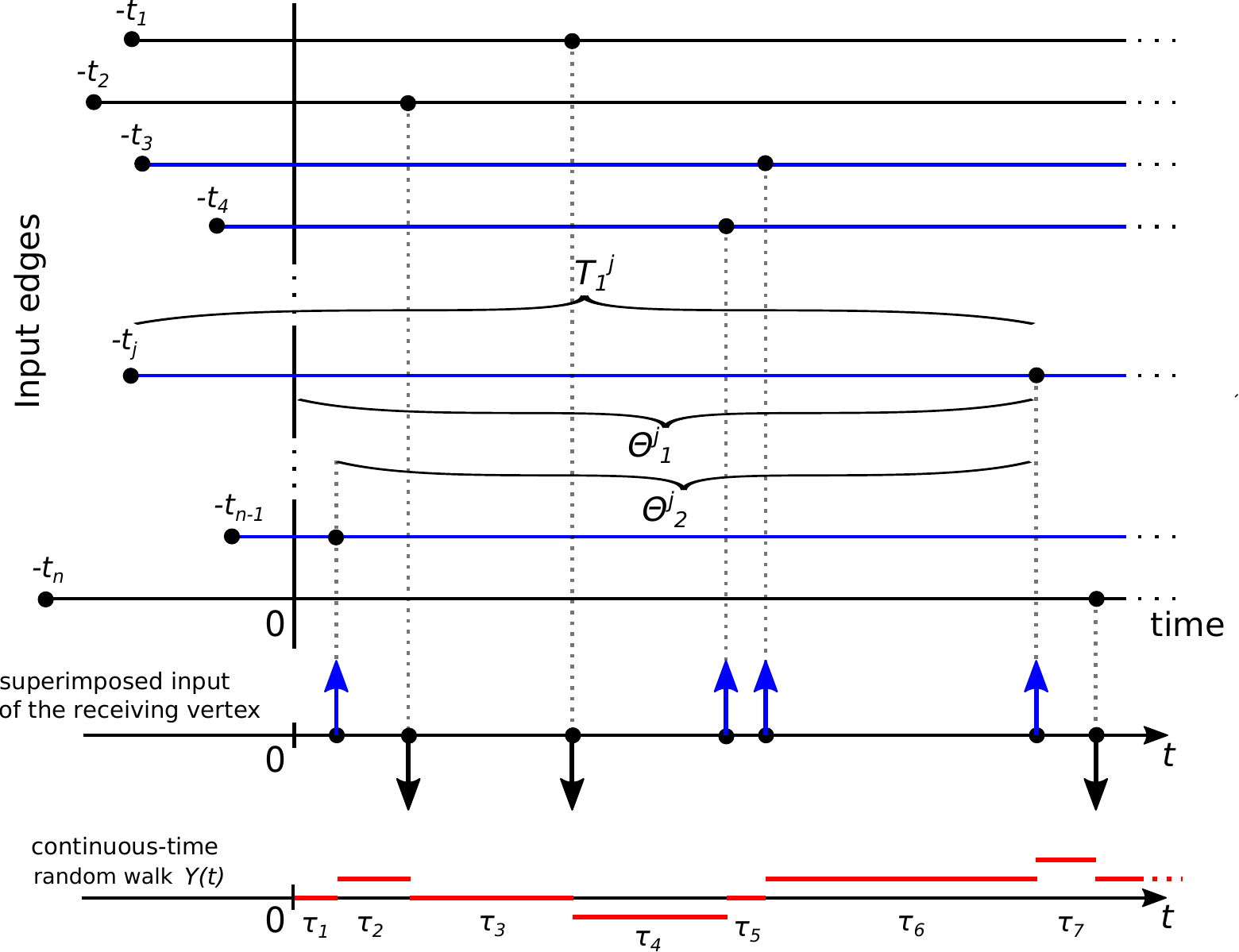}
	        \caption{\label{figi}Pooled input signal to the receiving vertex $v_0$. The figure shows how the input is formed by the superposition of the individual contributions coming from the input vertices $v_1,\dots,v_n$. The waiting times corresponding to vertices belonging to the subset $V^\uparrow$ are pictured in blue. Those related to vertices of $V^{\downarrow}$ are pictured in black. Note that the inter-arrival times $\tau_h$, $h \in \mathbb{N}^*$, of the continuous-time random walk $Y$ are given by $\tau_h=\min_{j\in \{1,\dots,n\}}\Theta_h^j$. We remark that, after the arrival of the first event, that is, that related to $v_{n-1}$, each component of the vector $\bm{\Theta_2}$ (i.e., for each input neuron, the remaining random waiting time between the arrival time of the first event and the subsequent one) is a forward recurrence time interval except that related to $v_{n-1}$ which is a complete inter-arrival time.}
	    \end{figure}
	    \label{seca}
	    Thus, we are concerned with a class of models of signal transmission on networks in which communication among vertices happens with the stochastic IF mechanism. This class of networks is frequently considered in computational neuroscience to model neuronal interactions (see e.g.\ \cite{Bernardi2021,Brunel2000183,Sanzeni2022,Spiridon2000529,Vogels200510786} and, for machine learning applications, \cite{maass1997networks,TAVANAEI201947,Zhang2017333}),
	    but, actually, its structure can be described without recurring to terminology relative to a specific field.
	    Indeed, we consider a directed graph $G=(V,E)$ where $V$ is the set of vertices and $E$ is the set of directed edges connecting the vertices. Further, the set $V$ is partitioned into two disjoint sets: $V= V^\uparrow \cup V^\downarrow$ characterized by different behaviours.
        On $G$ we focus on a mechanism of signal transmission among the vertices through their edges. Specifically, our goal is to recognize a transmission mechanism such that some specific features of the input can be reproduced in the output signal, i.e.\ input and output are consistent with respect to some distributional properties.
	    For this purpose, we start by fixing a reference vertex, say $v_0\in V$, and the vertices $v_1,v_2,\dots,v_n$, $n \in \mathbb{N}^* = \{ 1,2,\dots\}$ feeding $v_0$. The activity is represented by point processes associated to each of the vertices of the graph.
	    The characterization of each point process is given by the random inter-arrival times of the point signals relative to each vertex $v_i$, i.e. by sequences of positive real-valued random variables
	    $S^i = (S_j^i)_{j \in \mathbb{N}^*}$, $i \in \{0, \dots, n\}$.
	    The existence of an underlying network structure implies that the point processes are in general dependent (so that we are actually in presence of a multivariate point process). Furthermore, the inter-arrival times of each point process are also in general dependent. We will explicit the nature of all these dependencies in the following.
	    
	    With reference to Figure \ref{figi} which shows a sketch of the input signal, let us fix our observation starting time at $t=0$, which should coincide with the emission time and consequent reset of the reference vertex $v_0$.
	    The time $t=0$, for each point process, falls  inside or at the beginning of one of their inter-arrival intervals. We  rename this element of $S^i$ as $T_1^i$, for every $i \in \{0,\dots, n\}$ to underline that we actually refer to intervals containing the zero.
	    Call $\bm{T_1} = (T_1^1,\dots, T_1^n)$ the random vector of the $n$ positive inter-arrival times and call $-\bm{t}=(-t_1,\dots,-t_n)$, $t_i \in [0, \infty)$, $i\in \{1,\dots,n\}$, the vector of times at which the random inter-arrival times have started.
    	As soon as the first event of the input flow occurs, a new vector of random waiting times $\bm{T_2}$ is defined from $\bm{T_1}$: it coincides with $\bm{T_1}$ but we substitute the inter-arrival time corresponding to the vertex responsible of the event with its next inter-arrival time. We proceed with this rolling mechanism every time a new event occurs, thus defining $\bm{T_3}$, $\bm{T_4}$, and so forth. Figure \ref{figi} also shows the vectors of forward recurrence time intervals $\bm{\Theta_h} = (\Theta_h^1, \dots \Theta_h^n)$, $h \in \mathbb{N}^*$, associated to $\bm{T_h}$ (for the input neuron $v_j$, the quantity $\Theta_h^j$ represents the remaining random waiting time between the arrival time of the $(h-1)$-th event in the superimposed input flow and the subsequent event for $v_j$).

	    In Figure \ref{figi}, the input vertices are divided into those belonging to $V^\uparrow$ and those to $V^{\downarrow}$ (respectively pictured in blue and black) accounting for the effect they have on the receiving vertex. More precisely, the signal sent from one of the input vertices, say $v_i$, has value $c \in (0,\infty)$ if $v_i \in V^\uparrow$ while it is equal to $-c$ if $v_i \in V^\downarrow$.
	    The distinction between $\uparrow$-events and $\downarrow$-events determines a marked point process representing the superimposed input of the receiving vertex $v_0$ (shown in the figure).
	    Furthermore, we define a continuous-time random walk $Y = (Y(t))_{t \in \mathbb{R}_+}$, starting at zero at time $t=0$, whose jumps (of size $c$ or $-c$) happen in correspondence of the events of the marked point process.
	    Finally, the emission of a signal for the receiving vertex occurs when the random walk $Y$ attains a certain fixed level or threshold $b$.
	    Note that $Y(t) = b$ if and only if one of the input vertices belonging to $V^\uparrow$ has emitted a signal at time $t$. Similarly to the other vertices, the value of the signal emitted by $v_0$ will be $c$ or $-c$ depending on whether $v_0$ belongs to $V^\uparrow$ or $V^\downarrow$.
	    
	    So far we have described the model with a single reference vertex. This can be extended to $m$ reference vertices. These vertices may differ by their parameters such as the jump amplitude or the threshold level. Figure \ref{fig} shows a sketch of possible mesoscopic structures feeding two reference vertices $v_A$ and $v_B$ ($m=2$).   
	    The input vertices can be grouped in two possibly non-disjoint sets, $A=\{v_j\}_{j=1,\dots,\overline{n}}$ and $B=\{v_j\}_{j=\underline{n},\dots,n}$ with $\underline{n}\le\overline{n}+1$.
	    Note that $\bm{T_1}$, and hence all the other vectors derived from it, includes the inter-arrival times related to all the vertices in $A \cup B$. Only when necessary we will add a superscript describing the subset of interest: $\bm{T_1^A}$ or $\bm{T_1^B}$.
	    We explicitly underline that when the vector $\bm{T_1}$ will be updated to form $\bm{T_2}$, only one of the two subsets $\bm{T_2^A}$, $\bm{T_2^B}$ will differ from the previous subsets $\bm{T_1^A}$, $\bm{T_1^B}$. 
        Plainly, the cardinality of $A \cap B$ equals to $\overline{\underline{n}}= \overline{n}+1-\underline{n}$ and can be zero in case of disjointness of $A$ and $B$.
        The sets $A$ and $B$ of input vertices form the pooled input signals to $v_A$ and $v_B$, respectively. See Fig.\ \ref{fig}, right, for a visual representation of this configuration. At the level of the underlying structure it means that each vertex in $A$ has a directed edge to $v_A$ and each vertex in $B$ has a directed edge to $v_B$.
        The only assumption we consider concerning the microscopic relations between vertices in $A$ and in $B$ regards the presence of a dependency structure on the inter-arrival times between successive events.
        A similar configuration is present in \cite{rosenbaum2010pooling} in which the authors study the effect of correlations in the pooled input flow 
        and describe how output correlations in one layer impact correlations between the pooled inputs to the succeding layer in a feedforward network.         
        \begin{figure}
            \centering
            \begin{tabular}{ccc}
            \includegraphics[scale=.35]{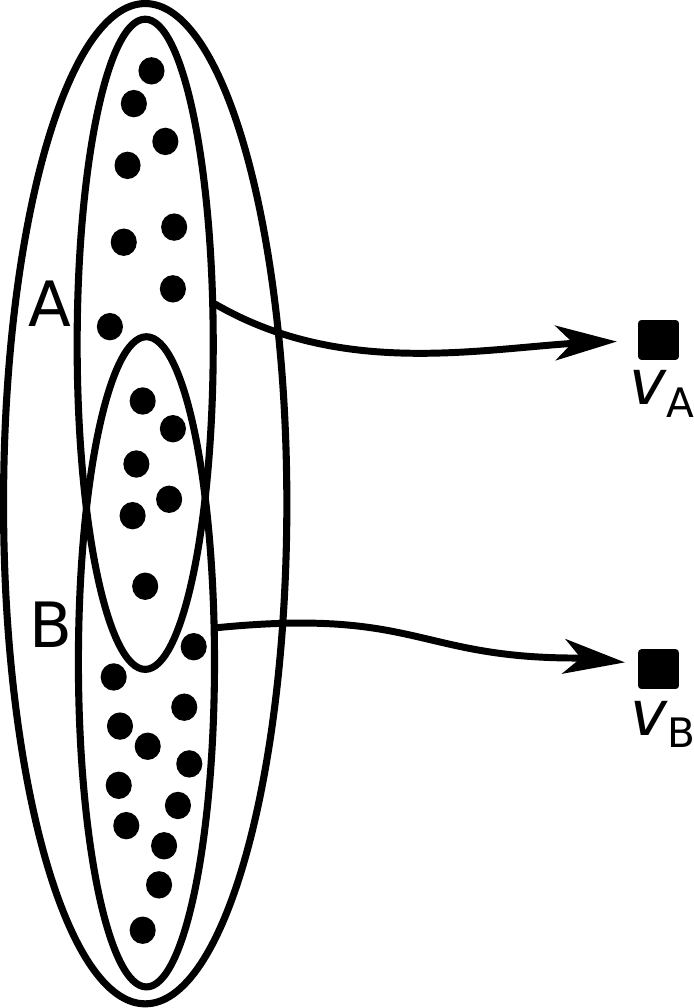}
            & \hspace{1cm} &
            \includegraphics[scale=.16]{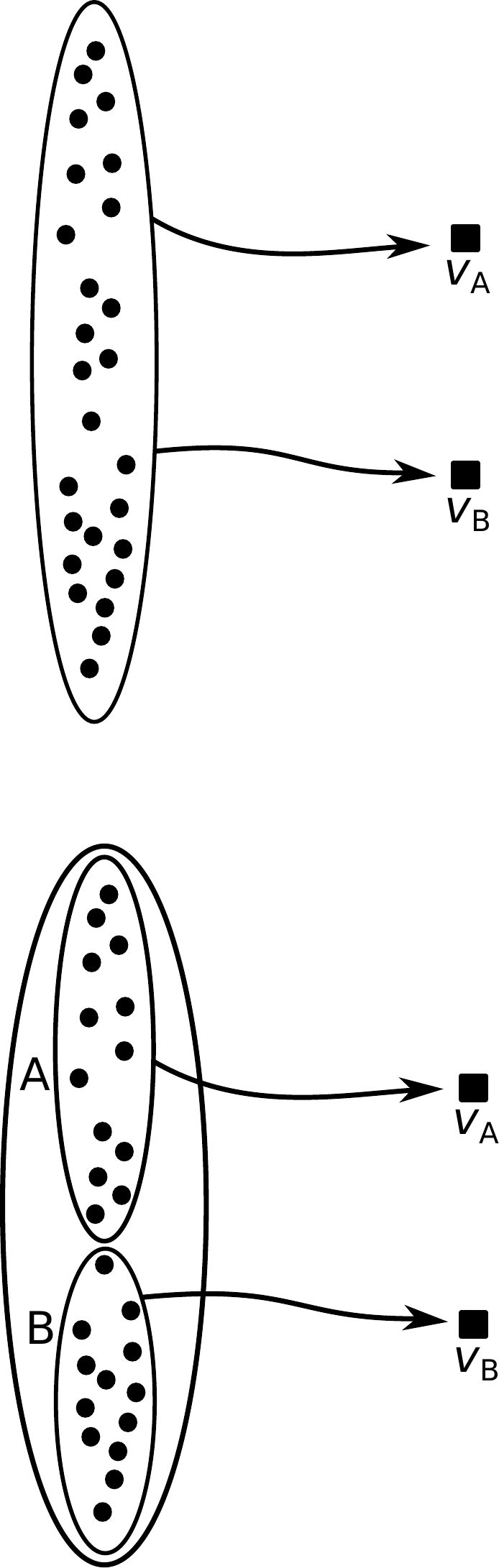}
            \end{tabular}
            \caption{\label{fig}In the left-hand-side picture, the mesoscopic structure of the input to $v_A$ and $v_B$ where $(\overline{n},\underline{n})=(14,9)$ and $(n,\overline{\underline{n}})=(26, 6)$.  The input vertices (black dots) contributing to the input signal of the two reference vertices $v_A$ and $v_B$ (black squares) are grouped into the two subsets $A$ and $B$. Set $A$ forms the pooled input for the reference vertex $v_A$ while set $B$ forms that for the reference vertex $v_B$. Besides the general dependence structure of the set of the $n$ input vertices, the dependence between the two pooled inputs comes also from the non-disjointness of $A$ and $B$. In the right-hand-side top picture, all input vertices directly contribute to both the reference vertices $v_A$ and $v_B$. This corresponds to the special case $(\overline{n},\underline{n})=(n,1)$. The dependence structure between the inter-arrival times of $v_A$ and $v_B$ follows from that of the inter-arrival times of the input vertices. Note that only if $v_A$ and $v_B$ share the same parameters, their output would be synchronized. In the right-hand-side bottom picture we illustrate the case of disjointness between $A$ and $B$. Here the dependence between $v_A$ and $v_B$ is inherited by the dependency structure between the inter-arrival times of all input vertices.}
        \end{figure}
	    

    
        From now on, to further detail the model, we will make use of  the terminology used in neuroscience. In particular, the vertices will correspond to neurons and the directed edges to their synaptic connections. Following the literature, neurons transform signals through the IF mechanism described in Section \ref{seca}. In this framework the events will correspond to the neuronal spikes and the inter-arrival times between events to the interspike intervals.
                
        Our aim is to show that a set of distributional features of the pooled input signal are retained in the output signal from the receiving neurons. We call this property \emph{input-output consistency}.
        To fix the distributional properties of ISIs we have to mathematically formalize the observed properties found in recorded data. In particular, various data sets show 
        the presence of heavy tails in the empirical distributions of the ISIs \cite{Gal20137912,gerstein1964random,kass2018computational,Tsubo2012,tuckwell1988introduction}. Further, due to the presence of connections between neurons, the ISIs clearly exhibit a dependence structure.
        
        We consider an ensemble of $n \in \mathbb{N}^*$ neurons, $v_1,\dots,v_{n}$, contributing to the pooled input signal of $m$ additional receiving neurons. For simplicity, and without losing generality, we fix $m= 2$ and call $v_A$ and $v_B$ the two receiving neurons. 
        In order to catch the heavy tails property of the empirical distributions and the dependencies between ISIs, we make four basic assumptions:
        \begin{itemize}
            \item[\textbf{H1}:] each of the vectors $\bm{T_1}$, $\bm{T_2}$, $\bm{T_3}, \dots$, is distributed as a regularly-varying random vector of parameter $\alpha \in (0,1)$ (see in the appendices, Definition \ref{def1});
            \item[\textbf{H2}:] all the one-dimensional marginals of such vectors are regularly-varying of the same order and asymptotically equivalent (see Definition \ref{asinteq});
            \item[\textbf{H3}:] each of the $n$ sequences $S^1,\dots, S^n$, describing the input ISIs is composed by pairwise asymptotically independent  random variables (see Definitions \ref{aaa} and \ref{aaaa});
            \item[\textbf{H4}:] input ISIs relative to different neurons (i.e.\ for each $h$, the components of $S^i$ and $S^j$ for $i \ne j$ belonging to the same $\bm{T_h}$) are fully dependent, i.e.\ for each $h$,
            \begin{align}
                \lim_{z \to \infty}\frac{\mathbb{P}(T_h^1 >z, \dots, T_h^n>z)}{\mathbb{P}(W_T>z)} = c_h \in (0, \infty),
            \end{align}
            where $W_T$ is a positive random variable asymptotically equivalent to the marginals of $\bm{T_h}$.
        \end{itemize}
          
        The rationale behind this modelling choice is to admit slowly-decaying tail behaviour with a possibly non-negligible but still mathematically tractable dependence structure.
        The basic definition of regularly-varying vectors (formula \eqref{f1}), imposes an asymptotically power-law probability distribution with a characterizing exponent $\alpha \in (0,1)$. This catches the experimentally observed tail of the empirical distribution of the ISIs.
        As a consequence of the regular variation hypothesis, a limiting measure $\nu$ on the spherical component of the random vector exists (formula \eqref{f2}). This measure describes the kind of dependence between the components of the vector (ISIs of different input neurons) for extremely long ISIs. Hence, the ISIs responsible of the long tails appear according to specific patterns regulated by the measure $\nu$. This would in turn imply that the measure $\nu$ could be involved in the neural coding.

	\section{Input-output consistency}\label{IOC}
	
	    In order to attain input-output consistency the hypotheses on the input must be reproduced in the output of the two receiving neurons $v_A$ and $v_B$.
    	A complete study of input-output consistency proceeds in two steps. First, we study of the model at the level of one of the two receiving neurons, then we analyse the full output consisting of the two receiving neurons.
    	
    	Let $\bm{T_1} = (T_1^1, \dots, T_1^{n})$ be the multivariate random vector describing the waiting time between signal emissions for each input neuron $v_1,\dots, v_{n}$ such that $T_1^j \ni 0$, $j \in \{1,\dots,n\}$.
    	As described above, $\bm{T_1}$ is distributed as a regularly-varying random vector of order $\alpha \in (0,1)$ (see Definition \ref{def1}). Recall also that we require each component of the vector to be regularly-varying of order $\alpha$ as well.
	
    	Consider now only the set $A$ characterized by $\overline{n}$ input neurons and the receiving neuron $v_A$. We have that the subvector $\bm{T_1^A}=(T_1^1,\dots, T_1^{\overline{n}})$ is regularly-varying distributed of order $\alpha$ as well.
    	In the following, for the sake of clarity we will drop the subscript in $v_A$ as well as the superscript in $\bm{T_1^A}$ and in all the derived vectors.
    	We also consider the already mentioned vector of forward recurrence time intervals $\bm{\Theta_1} = ( \Theta_1^1,\dots, \Theta_1^{\overline{n}} )$ which is defined, for each component (i.e.\ for each neuron belonging to $A$), as the remaining random waiting time between time zero and the following point-event.
    	
    	\begin{Theorem}\label{papap}
    	    The vector $\bm{\Theta_1}$ is regularly-varying of order $\alpha$ with asymptotically equivalent marginals.
	    \end{Theorem}
	       	
    	To determine the random waiting time until the next point-event in the input flow to $v$ we need to calculate the minimum between the components of the vector $\bm{\Theta_1}$:
    	\begin{align}
    	    \tau_1=\min_{j=1,\dots,\overline{n}}\Theta_1^j.
    	\end{align}
    	From \cite{MR2281913}, Section 5.4, we have that the minimum of dependent regularly-varying random variables is still regularly-varying of the same parameter. Hence $\tau_1$ is regularly-varying distributed of parameter $\alpha$.
     	The same reasoning as above and again considering \cite{MR2281913}, Section 5.4, allows us to state that
    	$\bm{\Theta_k} \sim \text{MRV}(\alpha)$ (multivariate regularly-varying vector, see Definition \ref{nyman}), for each given $k \ge 1$, and
    	\begin{align}\label{tastiera}
    	    \tau_k = \min(\Theta_k^1, \dots, \Theta_k^{\overline{n}}) \sim \text{RV}(\alpha), \qquad k \ge 1.
	    \end{align}
	    
	    Note that $(\tau_k)_{k\ge 1}$ is a sequence of dependent random variables, each of them regularly-varying distributed, and represents the waiting times between point-events in the pooled input flow for the receiving neuron $v$. Furthermore, the sequence $(\tau_k)_k$ is composed by asymptotically equivalent random variables.
	    Next theorem states the asymptotic independence for the sequence $(\tau_k)_k$.
        \begin{Theorem}\label{cpc}
            The sequence of random waiting times $(\tau_k)_k$ is pairwise asymptotically independent.
        \end{Theorem}
        The above theorem is a technical result that will be necessary to prove our next results. It seems acceptable from a modelling point of view. To realize this we have to imagine to conduct an experiment in which it can be possible to study pairs of $\tau$'s. In multiple realizations of the process $(\tau_k)_k$ it will be exceptional to obtain, in the same positions, repeated pairs of long-tailed ISIs.

        Now we turn our attention to the process $Y$ modelling the evolution of the membrane potential for the receiving neuron $v$. Notice that $v$ is a perfect integrator. Without losing generality, we set the jump amplitude of the membrane potential process $Y$ to $c=1$. 
        As soon as $Y$ reaches the constant boundary $b \in \mathbb{N}$, $v$ fires and $Y$ is reset to zero. Hence, we seek to study the distribution of the first passage time $Z$ of the membrane potential process $Y$ through $b$ (which represents the distribution of the first ISI of $v$).

        The random first passage time through the boundary $b$ equals the sum of a random number of the waiting times $\tau_k$'s. Specifically,
        \begin{align}
            Z= \sum_{k=1}^M \tau_k,
        \end{align}
        where $M$, independent of $(\tau_k)_k$, is the random time at which the discrete-time random walk $R=(R_n)_{n \ge 0}$, embedded into $Y$, attains the level $b$.
        
        The random walk $R$ is in general a biased random walk in which the bias is given by the nature of the input neurons: if the number of excitatory neurons exceeds the number of inhibitory neurons then $p$, i.e.\ the probability of a positive jump is larger than $1/2$ and there is a drift towards the boundary. In this case $M$ is a genuine random variable (i.e.\ $M$ is finite a.s.) and $\mathbb{E}M<\infty$. Otherwise, if there is an excess of inhibition ($p<1/2$) the drift leads the walker far from the boundary. In this case $M$ is defective (i.e.\ $\mathbb{P}(M < \infty) < 1$) and there is a non-negligible probability mass located at infinity (meaning that there exist trajectories of $R$ which eventually will not attain the boundary $b$) (see \cite{ricciardi}, chapter IV.1, for biased random walks).
        From a pure modelling point of view such trajectories will not give rise to any firing of the neuron $v$ (i.e.\ $v$ will remain silent). This behaviour is not interesting for us as we want to study input-output consistency of neuronal ISIs from active neurons. Hence, to let the analysis be meaningful, instead of looking at the distribution of $M$ we rather look at the distribution of $M$ conditional on the event $\{ M<\infty \}$. That is to say we condition to attaining the boundary in some finite future time. In this way the conditional expectation $\mathbb{E}[M|M<\infty]$ is finite as well. However, in order to simplify the notation, in the following we will frequently omit the conditioning if $p < 1/2$. 
        If we have exactly $p=1/2$, $M$ is almost surely finite but
        $\mathbb{E}M = \infty$.
        In our study we disregard this case as it is not interesting from a modelling point of view. Indeed, having $p=1/2$, that is to say, the same number of excitatory and inhibitory neurons is in fact exceptional.  

        Next theorem describes the distribution of $Z$ in the case when $p \ne 1/2$, i.e.\ in the cases in which the expectation of $M$ for $p >1/2$ (or the mentioned conditional expectation of $M$ for $p < 1/2$) is finite. 
 
        \begin{Theorem}\label{vavava}
            For $p \ne 1/2$,
            the random variable $Z$ representing the first ISI of the receiving neuron is regularly-varying distributed with order $\alpha$.
        \end{Theorem}
        
        Consider now the sequence of output ISIs of the receiving neuron: 
        \begin{align}\label{ziai}
            (Z_i)_{i \in \mathbb{N}^*} = \Bigl(\sum_{h=1}^{M_1}\tau_h, \sum_{h=M_1+1}^{M_2}\tau_j,\dots, \sum_{h=M_{r-1}+1}^{M_r}\tau_h, \dots\Bigr).
        \end{align}
        As far as \textbf{H2} is concerned, note that each $Z_i$, $i \in \mathbb{N}^*$, composed by the random number $M_i - M_{i-1}$ of terms (let $M_0 = 0$ a.s.), can also be proven to be regularly-varying distributed of order $\alpha$, similarly as done for $Z$ above.
        Our next aim is to prove that the sequence \eqref{ziai}        
        is composed by pairwise asymptotically independent random variables, and hence \textbf{H3} is verified for the output of the receiving neuron. 
        For the sake of clarity, we show first a result on the asymptotic independence of the sum of two random variables taken from a group of three. This will be preparatory for the proof of the complete Theorem \ref{babel} below.
        
        \begin{Proposition}
            Let $X,Y,Z$ be three pairwise asymptotically independent, asymptotically equivalent, positive real-valued random variables, such that each of them is regularly-varying distributed of parameter $\alpha \in (0,1)$. Then, the random variables $S = X + Y$ and $Z$ are asymptotically independent.
        \end{Proposition}
        \begin{proof}
            The asymptotic equivalence hypothesis means that there exists a positive random variable $\xi$ such that
            \begin{align}
                \lim_{x \to \infty}\frac{\mathbb{P}(X>x)}{\mathbb{P}(\xi>x)} = c_1 >0,
                \quad
                \lim_{x \to \infty}\frac{\mathbb{P}(Y>x)}{\mathbb{P}(\xi>x)} = c_2 >0,
                \quad
                \lim_{x \to \infty}\frac{\mathbb{P}(Z>x)}{\mathbb{P}(\xi>x)} = c_3 >0,
            \end{align}
            with $c_1,c_2,c_3$ finite constants. Moreover, by Lemma 3.1 of \cite{MR2281913}
            \begin{align}
                \lim_{x \to \infty}\frac{\mathbb{P}(S>x)}{\mathbb{P}(\xi>x)} = c_1+c_2.
            \end{align}
            Then,
            \begin{align}
                & \lim_{x \to \infty}\frac{\mathbb{P}(S>x,Z>x)}{\mathbb{P}(\xi>x)} \\
                &= \lim_{x \to \infty}\frac{\mathbb{P}(S>x,Z>x|X \ge Y)\mathbb{P}(X \ge Y)+\mathbb{P}(S>x,Z>x|Y \ge X)\mathbb{P}(Y \ge X)}{\mathbb{P}(\xi>x)} \notag \\
                &\le \lim_{x \to \infty}\frac{\mathbb{P}(X>x/2,Z>x|X \ge Y)\mathbb{P}(X \ge Y)+\mathbb{P}(Y>x/2,Z>x|Y \ge X)\mathbb{P}(Y \ge X)}{\mathbb{P}(\xi>x)} \notag \\
                &\le \lim_{x \to \infty}\frac{\mathbb{P}(X>x/2,Z>x/2)+\mathbb{P}(Y>x/2,Z>x/2)}{\mathbb{P}(\xi>x)} \notag \\
                & = \lim_{x \to \infty} \left( \frac{\mathbb{P}(X>x/2,Z>x/2)}{\mathbb{P}(\xi>x/2)}+\frac{\mathbb{P}(Y>x/2,Z>x/2)}{\mathbb{P}(\xi>x/2)}\right) \frac{\mathbb{P}(\xi>x/2)}{\mathbb{P}(\xi>x)}. \notag
            \end{align}
            Now, note that both addends in parentheses in the above formula converge to zero, as $x$ goes to infinity, due to the pairwise asymptotic independence of  $X$, $Y$, and $Z$ (see Definition \ref{aaa}). Further, since
            \begin{align}
                \lim_{x \to \infty} \frac{\mathbb{P}(\xi>x/2)}{\mathbb{P}(\xi>x)}
                = \lim_{x \to \infty} \frac{\mathbb{P}(X>x/2)}{\mathbb{P}(X>x)} = 2^\alpha < 2,
            \end{align}
            for the regular variation of $X$,
            we obtain that $Z$ and $S$ are asymptotically independent.
        \end{proof}

        \begin{Theorem}\label{babel}
            For $p \ne 1/2$, the sequence $(Z_i)_{i \in \mathbb{N}^*}$ of successive ISIs of the receiving neuron is composed by pairwise asymptotically independent random variables.
        \end{Theorem}

        Hence, the results contained in Theorems \ref{vavava} and \ref{babel} show that the output of the single receiving neuron $v$ is consistent with \textbf{H2} and \textbf{H3} on the input signal. 

        We move now to the complete problem of proving input-output consistency for a vector of receiving neurons. Specifically, we need to prove \textbf{H1} and \textbf{H4} for the output of the receiving neurons. To do so it is enough proving it for a vector of size two.

        We start with \textbf{H1}, proving \eqref{duo} for the vector of the output ISIs $(Z_A,Z_B)$ relative to both receiving neurons $v_A$ and $v_B$. \textbf{H1} is proven by the following theorem:
        \begin{Theorem}\label{tazzina}
            The  vector $(Z_{A,j},Z_{B,k})$ with fixed $j,k \in \mathbb{N}^*$ of output ISIs including a common firing event (i.e.\ the marginals are temporally superimposed) is distributed as a regularly-varying random
            vector of parameter $\alpha$.
        \end{Theorem}

        Furthermore, next theorem proves that $(Z_{A,j},Z_{B,k})$ exhibits full dependence (i.e.\ \textbf{H4} holds).
        \begin{Theorem}\label{ark}
            The  vector $(Z_{A,j},Z_{B,k})$ with fixed $j,k \in \mathbb{N}^*$ of output ISIs including a common firing event
            shows full dependence in the sense of hypothesis \textbf{H4}. 
        \end{Theorem}

    \section{Conclusions}\label{conc}
    
        We have described a neuronal model characterized by ISIs distributions exhibiting power-law decaying tails. Our results show that the IF paradigm is able to transform the input to neurons into an output preserving the dependence structure and the asymptotic behaviour of the ISIs distributions.
        The check of this consistency is disregarded in many modelling instances, but it seems unavoidable wishing to consider a model in which the output of some neurons becomes the input for others. Our description differs from others in that we have privileged asymptotic features of the ISIs distribution using regularly-varying vectors and the property of asymptotic independence.
        The latter concept is introduced for technical reasons to prove the wished results. However, it seems acceptable that ISIs of the same spike train become asymptotically independent. Indeed, it is highly unlikely for long ISIs to be shortly followed by other long ISIs.
        Other studies focus on a more quantitative analysis. For this goal they simplify the description of the dependence to the evaluation of correlations \cite{rosenbaum2010pooling}. 

        The considered model is clearly over-simplified with respect to many others appeared in the literature, as it disregards spontaneous decay of the membrane potential, the existence of reversal potentials and many other features. However, it  allows to introduce dependencies between ISIs of each spike train and between ISIs of different spike trains, a property generally studied only through simulations.

        This work is a first step in the study of the consistency in the network and focuses on the transmission in the internal part of the network. Actually, the considered network structure is of feedforward type. Improvements of the model should include more properties of neurons, retaining the consistency between input and output.
        The biological advantage of ISIs with heavy-tailed distributions may be a subject of discussion. We limit ourselves to mention an increasing of robustness of the code related to a decreased sensibility to the possible loss of a single long ISI.
        
        There are instances in which a first layer of neurons receives external inputs that could be Poissonian (see e.g.\ \cite{Duchamp-Viret200597}).
        In similar cases it may happen that successive layers of neurons transform the input generating an output that is characterized by heavy tails, which in turn feeds next layers following the dynamics described in the present paper. Hence, we plan future studies focusing on recognizing possible paradigms determining this transformation.
        Moreover, the present model considers only the tail behaviour of ISIs distributions. We plan to make use of simulations to investigate different patterns generated by models having the described asymptotic properties.

        We add that the interest of the obtained results is not limited to applications to neuroscience but they could result interesting also in the case of networks used in machine learning and artificial intelligence.

    \subsubsection*{Acknowledgements}
    
        L.\ Sacerdote and F.\ Polito thank B.\ Lindner for his remarks on a preliminary version of the work that helped us to improve the paper.

	\begin{appendices}
	
	\section{Mathematical background}
	
	    \label{appe}	
	    Let us describe here the basic mathematical tools we made use of for the description of the model and of its properties. Recall that we are interested mainly in the tail behaviour of random vectors.
	    As basic references for the definitions and properties listed here we suggest for instance to consult \cite{basrak2000sample,MR1925445,MR2281913}. 
	    
	    We denote by $\mathbb{S}^{d-1}$, $d \in \mathbb{N}$, the $(d-1)$-dimensional unit sphere embedded in $\mathbb{R}^d$, that is for example $\mathbb{S}^0 = \{-1,1\}$, $\mathbb{S}^1$ is the unit circle, and so forth. Furthermore we denote by $\mathcal{B}(\mathbb{S}^{d-1})$ the Borel sigma algebra on $\mathbb{S}^{d-1}$.
	    
	    \begin{Definition}
            \label{def1}
    	    Given a probability space $(E,\mathcal{F},\mathbb{P})$, a $d$-dimensional random vector $\bm{X}$,
            defined on it and taking values in $\mathbb{R}^d \backslash \{\bm{0}\}$ is said to be \emph{regularly-varying of order $\alpha$}, $\alpha \ge 0$ if there exist a probability measure $\nu$ on the unit $(d-1)$-dimensional sphere such that for
            any strictly positive $t$ and any set $S \in \mathcal{B}(\mathbb{S}^{d-1})$, it holds that
            \begin{align}
                & \label{f1} \lim_{x \to \infty} \frac{\mathbb{P}(\| \bm{X}\|>tx)}{\mathbb{P}(\|\bm{X}\|>x)}= t^{-\alpha}, \\
                & \label{f2} \lim_{x \to \infty} \mathbb{P} \left( \left.\frac{\bm{X}}{\|\bm{X}\|}\in S\right| \|\bm{X}\|>x\right) = \nu(S),
            \end{align}
            where $\|\cdot \|$ is a given norm on $\mathbb{R}^d$.
	    \end{Definition}
	    
	    \begin{Remark}
	        \label{remcoe}
	        Definition \ref{def1} can be rewritten in an equivalent form which is more frequently given in the literature. Formulae \eqref{f1} and \eqref{f2} hold if and only if for
            every strictly positive $t$ and any set $S \in \mathcal{B}(\mathbb{S}^{d-1})$,
	        \begin{align}
	            \label{coelho}
	            \lim_{x \to \infty} \frac{\mathbb{P}\left(\| \bm{X}\|>tx, \frac{\bm{X}}{\|\bm{X}\|}\in S\right)}{\mathbb{P}(\|\bm{X}\|>x)} = t^{-\alpha} \nu(S).
	        \end{align}
	        Note that formula \eqref{coelho} defines a limiting infinite measure $\mu$ on $(\mathbb{R}^d\backslash \{\bm{0}\}, \mathcal{B}(\mathbb{R}^d\backslash \{\bm{0}\}))$ through the $\pi$-system $\mathcal{C}=\{A_{t,S}\}$, $S \subseteq \mathbb{S}^{d-1}$, $t > 0$, where $A_{t,S}= (t,\infty)\times S$, and $\mu(A_{t,S})= t^{-\alpha}\nu(S)$.
	    \end{Remark}
	    
	    \begin{Remark}
	        Interestingly enough, from Definition \ref{def1} and Remark \ref{remcoe}, the limiting measure $\mu$ factorizes into the two components \eqref{f1} and \eqref{f2}. This relates to the fact that, in the tails, the radial and spherical parts of the $d$-dimensional vector $\bm{X}$ become independent.
	    \end{Remark}
	    
	    We describe now, an alternative and equivalent definition which proves to be useful when dealing with random vectors made up of forward recurrence time intervals of regularly-varying waiting times. For it, we refer to \cite{basrak2000sample,MR900810}.
	    
        \begin{Definition}
	        \label{nyman}

            The random vector $\bm{X}$ taking values in $[\bm{0},\bm{\infty})$, with joint distribution function $F$, is said to be regularly-varying if it exists a Radon measure $\mu$ on $[\bm{0},\bm{\infty}]\backslash \{\bm{0}\}$ such that
            \begin{align}\label{duo}
                \lim_{x \to \infty} \frac{1-F(x\bm{t})}{1-F(x\bm{1})} = \lim_{x \to \infty} \frac{\mathbb{P}(x^{-1}\bm{X} \in [\bm{0},\bm{t}]^c)}{\mathbb{P}(x^{-1}\bm{X} \in [\bm{0},\bm{1}]^c)} = \mu ([\bm{0},\bm{t}]^c)
            \end{align}
            for every $\bm{t}\in[\bm{0},\bm{\infty})\backslash \{\bm{0}\}$ point of continuity for $\mu([\bm{0},\cdot]^c)$.
	    \end{Definition}
	    
	    \begin{Remark}
	        In dimension one, from Definition \ref{def1} or formula \eqref{coelho}, a real-valued random variable $X$ is said to be regularly-varying of order $\alpha \ge 0$ if and only if for every $t>0$, and every subset $S$ of $\mathbb{S}^0 = \{-1,1\}$,
	        \begin{align}
	            \lim_{x \to \infty} \frac{\mathbb{P}(|X|>tx, X/|X| \in S)}{\mathbb{P}(|X|>x)} = t^{-\alpha}\nu(S),
	        \end{align}
	        actually meaning that the following hold:
	        \begin{align}
	            \label{trio}
	            & \lim_{x \to \infty} \frac{\mathbb{P}(|X|>tx)}{\mathbb{P}(|X|>x)}=t^{-\alpha}, \quad
	            \lim_{x \to \infty} \frac{\mathbb{P}(X> tx)}{\mathbb{P}(|X|>x)} = pt^{-\alpha}, \quad
	            \lim_{x \to \infty} \frac{\mathbb{P}(X<- tx)}{\mathbb{P}(|X|>x)}= qt^{-\alpha},
	        \end{align}
	        where $p= \nu(\{1\})$, $q=\nu(\{-1\})$.
	        
	        If $X$ is almost surely strictly positive then formulae \eqref{trio} and \eqref{duo} boil down to checking that
	        \begin{align}
	            \lim_{x \to \infty} \frac{\mathbb{P}(X > tx)}{\mathbb{P}(X>x)} = t^{-\alpha}, \qquad \forall \: t>0.
	        \end{align}
	    \end{Remark}

        \begin{Definition}\label{asinteq}
            Let $X$, $Y$ be real-valued, positive random variables. We say they are asymptotically equivalent if
            \begin{align}
                \lim_{x \to \infty}\frac{\mathbb{P}(X > x)}{\mathbb{P}(Y > x)} = c \in (0,\infty).
            \end{align}
        \end{Definition}	    
        \begin{Definition}\label{aaa}
            Let $(X_i)_{i=1}^n$ be a finite sequence of positive and asymptotically equivalent random variables. Let $W$ be a positive random variable such that
            \begin{align}
                \frac{\mathbb{P}(X_i>x)}{\mathbb{P}(W>x)} \to c_i \in (0, \infty),
                \qquad i \in \{1,\dots,n\},
            \end{align}
            as $x \to \infty$. If
            \begin{align}
                \frac{\mathbb{P}(X_j>x,X_k>x)}{\mathbb{P}(W>x)} \to 0,
                \qquad \text{for every } j\ne k,
            \end{align}
            as $x \to \infty$,
            then the random variables $(X_i)_{i=1}^n$ are said to be pairwise asymptotically independent (or upper-tail independent).
        \end{Definition}
        \begin{Definition}\label{aaaa}
            Let $(X_i)_{i=1}^\infty$ be an infinite sequence of positive random variables. If the subsequences $(X_{i_1}, \dots, X_{i_k})$, for every possible choice of the $k$-tuple $(i_1, \dots, i_k)$,  for every $k \in \{2, 3,\dots\}$,  are composed by pairwise asymptotically independent random variables, then the random variables $(X_i)_{i=1}^\infty$ are said to be pairwise asymptotically independent (or upper-tail independent).
        \end{Definition}

    \section{Proofs}
	
	   \begin{proof}[\textbf{Proof of Theorem \ref{papap}}]
    	    To check that formula \eqref{duo} is valid for $\bm{\Theta_1}$, fix $\bm{s} \in \mathbb{R}_+^{\overline{n}}\backslash \{\bm{0}\}$, and for every $a \in \mathbb{R}_+^*$ consider the ratio
	        \begin{align}
	            \frac{1-\mathbb{P}(\bm{\Theta_1}\le a \bm{s})}{1-\mathbb{P}(\bm{\Theta_1} \le a \bm{1})} & = \frac{\mathbb{P}(\bm{\Theta_1} \in [\bm{0},a\bm{s}]^c)}{\mathbb{P}(\bm{\Theta_1} \in [\bm{0},a\bm{1}]^c)} \\
	            & = \frac{\mathbb{P}(\bm{T_1} \in [\bm{0},a\bm{s}+\bm{t}]^c | \bm{T_1}>\bm{t})}{\mathbb{P}(\bm{T_1} \in [\bm{0},a\bm{1}+\bm{t}]^c | \bm{T_1}>\bm{t})}, \notag
	        \end{align}
	        where we have considered the relationship between $\bm{\Theta_1}$ and $\bm{T_1}$. Hence,
	        \begin{align}
	            \frac{1-\mathbb{P}(\bm{\Theta_1}\le a \bm{s})}{1-\mathbb{P}(\bm{\Theta_1} \le a \bm{1})} & = \frac{\mathbb{P}(\bm{T_1} \in [\bm{0},a\bm{s}+\bm{t}]^c , \bm{T_1}>\bm{t})}{\mathbb{P}(\bm{T_1} \in [\bm{0},a\bm{1}+\bm{t}]^c , \bm{T_1}>\bm{t})} \\
	            & = \frac{\mathbb{P}(\bm{T_1}\in[\bm{t},\bm{\infty})\backslash [\bm{t}, a\bm{s}+ \bm{t}])}{\mathbb{P}(\bm{T_1}\in[\bm{t},\bm{\infty})\backslash [\bm{t}, a\bm{1}+ \bm{t}])} \notag \\
	            & =\frac{\mathbb{P}(\bm{T_1}-\bm{t}\in[\bm{0},\bm{\infty})\backslash [\bm{0}, a\bm{s}])}{\mathbb{P}(\bm{T_1}-\bm{t}\in[\bm{0},\bm{\infty})\backslash [\bm{0}, a\bm{1}])} \notag \\
	            & = \frac{1-\mathbb{P}(\bm{T_1}-\bm{t} \le a\bm{s})}{1-\mathbb{P}(\bm{T_1}-\bm{t} \le a\bm{1})}. \notag
	        \end{align}
            Applying Lemma 3.12 of \cite{MR2281913} with $\bm{X_1} = \bm{T_1}$ and $\bm{X_2} = - \bm{t}$, it is immediate to conclude that $\bm{\Theta_1}$ is regularly-varying distributed of order $\alpha$.
    	    Moreover, its marginals are asymptotically equivalent being shifts of the originary inter-arrival times.
    	\end{proof}
    	
		\begin{proof}[\textbf{Proof of Theorem \ref{cpc}}]
			To prove the property we need to show that, for each $h,j \in \mathbb{N}^*$, 
	    	\begin{align}\label{teiera}
	    		\frac{\mathbb{P} (\tau_h>z, \tau_{h+j} > z)}{\mathbb{P} (\tau_h>z)} \underset{z \to \infty}{\longrightarrow} 0.
	    	\end{align}
	    	
	    	We first fix the necessary notation: let $v_{m_1}$ and $v_{m_2}$ be the neurons responsible of the $h$-th and $(h+j)$-th firing events, respectively; denote by $\mathcal{Q}^{m_1}$ and $\mathcal{P}^{m_1}$ the inter-arrival times of $v_{m_1}$ containing respectively $\Theta_h^{m_1}$ and $\Theta_{h+j}^{m_1}$.
	    	Note that $\mathcal{Q}^{m_1}$ and $\mathcal{P}^{m_1}$ do not coincide and both belong to $S^{m_1}$ which is
		    a sequence of pairwise asymptotic independent and marginally regularly-varying random variables of order $\alpha$.
	    	A similar notation is used for the analogous quantities $\mathcal{Q}^{m_2}$ and $\mathcal{P}^{m_2}$ relative to the neuron $v_{m_2}$.
	    	
	    	Our first step consists in studying the numerator of \eqref{teiera}. We rewrite the joint survival function introducing a conditioning on what happens to the forward recurrence time intervals $(\Theta_h^{m_2}, \Theta_{h+j}^{m_2})$ relative to the neuron responsible of the $(h+j)$-th firing event.  We have
			\begin{align}\label{biquattro}
				\mathbb{P} (\tau_h>z, \tau_{h+j} > z) = \mathbb{P}(\Theta_h^{m_2}>z, \Theta_{h+j}^{m_2}>z)p^{m_2}_{h,h+j}(z),
			\end{align}
			where
			\begin{align}
				p^{m_2}_{h,h+j}(z) &= \mathbb{P} \Bigl(\Theta_h^1>z, \dots, \Theta_h^{m_2-1}>z,\Theta_h^{m_2+1}>z,\dots, \Theta_h^{\overline{n}}>z , \\
				&\Theta_{h+j}^1>z, \dots, \Theta_{h+j}^{m_2-1}>z, \Theta_{h+j}^{m_2+1}>z,\dots, \Theta_{h+j}^{\overline{n}}>z | \Theta_h^{m_2}>z, \Theta_{h+j}^{m_2}>z\Bigr). \notag
			\end{align}
			Let $\mathcal{W}$ be a suitable r.v.\  controlling the tail of the components of $S^{m_2}$, we can write that
			\begin{align}\label{zero}
			\frac{\mathbb{P}(\tau_h>z,\tau_{h+j}>z)}{p^{m_2}_{h,h+j}(z)\mathbb{P}(\mathcal{W}>z)}
				= \frac{\mathbb{P}(\Theta_h^{m_2}>z, \Theta_{h+j}^{m_2}>z)}{\mathbb{P}(\mathcal{W}>z)} \le \frac{\mathbb{P}(\mathcal{Q}^{m_2}>z, \mathcal{P}^{m_2}>z)}{\mathbb{P}(\mathcal{W}>z)},
			\end{align}
			where the first equality comes by applying equation \eqref{biquattro} and the last inequality uses the stochastic ordering between the inter-arrival times $(\mathcal{Q}^{m_2},\mathcal{P}^{m_2})$ and the corresponding forward recurrence time intervals $(\Theta_h^{m_2},\Theta_{h+j}^{m_2})$.  
			Recalling that $S^{m_2}$ is a sequence of pairwise asymptotically independent random variables, the right-hand side of \eqref{zero} tends to zero as $z \to \infty$. Hence,
			\begin{align}\label{due}
			    \frac{\mathbb{P}(\tau_h>z,\tau_{h+j}>z)}{p^{m_2}_{h,h+j}(z)\mathbb{P}(\mathcal{W}>z)}     \underset{z \to \infty}{\longrightarrow} 0.
	    	\end{align}
	    	
	    	Our second step considers the tail behaviour of the denominator of \eqref{teiera}, i.e.\ of the marginal survival function $\mathbb{P} (\tau_h>z)$.
	    	Note that $\mathcal{W}$, given the asymptotic equivalence condition, controls also the tail of the components of $S^{m_1}$. Hence, for constants $k_h>0$, using the asymptotic equivalence between $\mathcal{Q}^{m_1}$ and $\Theta_h^{m_1}$, we have,
			\begin{align}\label{uno}
				k_h = \lim_{z \to \infty} \frac{\mathbb{P}(\mathcal{Q}^{m_1}>z)}{\mathbb{P}(\mathcal{W}>z)}=\lim_{z \to \infty} \frac{\mathbb{P}(\Theta_h^{m_1}>z)}{\mathbb{P}(\mathcal{W}>z)} =\lim_{z \to \infty} \frac{\mathbb{P}(\tau_h > z)}{\mathbb{P}(\mathcal{W}>z) p_{h}^{m_1}(z)},
			\end{align}
			where
			\begin{align}\label{aa}
				p_h^{m_1}(z) = \mathbb{P}(\Theta_h^1>z, \dots, \Theta_h^{m_1-1}>z,\Theta_h^{m_1+1}>z,\dots, \Theta_h^{\overline{n}}>z | \Theta_h^{m_1}>z).
			\end{align}
			Considering \eqref{uno} and \eqref{due},
	    	we have that for each $h,j \in \mathbb{N}^*$, 
	    	\begin{align}
	    		\frac{\mathbb{P} (\tau_h>z, \tau_{h+j} > z)}{\mathbb{P} (\tau_h>z)} \underset{z \to \infty}{\longrightarrow} 0.
	    	\end{align}
	    	This result follows by noting that  $p_h^{m_1}(z)/p_{h,h+j}^{m_2}(z)$ does not converge to zero as $z \to \infty$.
	    	Indeed, the denominator is bounded by one, while the numerator \eqref{aa} cannot converge to zero as the waiting times for different neurons, and hence the corresponding forward recurrence time intervals, are taken to be fully dependent (\textbf{H4}).
	    \end{proof}

        \begin{proof}[\textbf{Proof of Theorem \ref{vavava}}]
        
            For the proof we will mainly rely on Theorem 2 of \cite{MR2745413} giving the tail behaviour of series of randomly weighted asymptotically independent random variables. 
            This is possible as we can rewrite the random sum defining $Z$ as an infinite sum with random coefficients,
            \begin{align}\label{cc}
                Z= \sum_{k=1}^M \tau_k = \sum_{k=1}^\infty \tau_k \xi_k,
            \end{align}
            where $\xi_k = \mathbb{1}_{\{1,\dots,M\}}(k)$ is a sequence of dependent Bernoulli random variables, independent of the sequence $(\tau_k)_k$.
            For each $k \in \mathbb{N}^*$, denote by $\bar{F}_k(t)$, the survival function $\mathbb{P}(\tau_k > t)$ of $\tau_k$. We check now the hypotheses of Theorem 2 of \cite{MR2745413}. We note that the assumptions $A1$ and $A2$ are trivially verified. Moreover, since $(\tau_k)_k$ is a sequence of dependent random variables each of them regularly-varying distributed of common order $\alpha$ (see formula \eqref{tastiera}), we have that $\mathbb{J}^-_{F_k}$ is positive as
            \begin{align}
                & \mathbb{J}^-_{F_k} = -\lim_{y \to \infty}\frac{\log \left(\limsup_{t \to \infty} \frac{\bar{F}_k(ty)}{\bar{F}_k(t)}\right)}{\log y} = \alpha, \\
                & \mathbb{J}^+_{F_k} = -\lim_{y \to \infty}\frac{\log \left(\liminf  f_{t \to \infty} \frac{\bar{F}_k(ty)}{\bar{F}_k(t)}\right)}{\log y} = \alpha.
            \end{align}

            We are left to check the moment condition on the sequence $(\xi_k)_k$. Since it is a Bernoulli sequence, for every $b \in \mathbb{R}_+^*$, $k \in \mathbb{N}^*$, we have $\mathbb{E}\,\xi_k^b = \mathbb{E}\,\xi_k$. It follows that, for $\delta > 0$, and recalling that $\alpha \in (0,1)$,
            \begin{align}\label{aax}
                \sum_{k=1}^\infty \mathbb{E}\,\xi_k^{\alpha-\delta}
                = \sum_{k=1}^\infty \mathbb{E}\,\xi_k^{\alpha+\delta} = \sum_{k=1}^\infty \mathbb{E}\,\xi_{k} 
                = \sum_{k=1}^\infty \mathbb{P}(M \ge k)
                = \mathbb{E} M.
            \end{align}
            The finiteness of $\mathbb{E}M$ for $p > 1/2$ or of the conditional expectation $\mathbb{E}[M|M<\infty]$ for $p < 1/2$ ensures the fineteness of \eqref{aax}.

            Since in our case
            \begin{align}
                L = \bigwedge_k \lim_{y \to 1+} \left(\liminf_{t \to \infty} \frac{\bar{F}_k(ty)}{\bar{F}_k(t)}\right) = 1,
            \end{align}
            Theorem 2 of \cite{MR2745413} implies that $\mathbb{P}(Z > x) \sim \sum_{k=1}^\infty \mathbb{P}(\tau_k \xi_k> x)$.

            In the final step we prove regular variation for the latter series. Notice that $\mathbb{P}(\tau_k \xi_k > x) = \mathbb{P}(\xi_k =1) \mathbb{P}(\tau_k>x) = \mathbb{P}(M \ge k) \mathbb{P}(\tau_k>x) \le \mathbb{P}(M \ge k)$. Since $p \ne 1/2$, this leads to uniform convergence of $\sum_{k=1}^\infty \mathbb{P}(\tau_k \xi_k> x) \le \mathbb{E} M < \infty$, for every $x \in \mathbb{R_+}$, $\alpha \in (0,1)$. Further, since
            \begin{align}
                &\lim_{x \to \infty} \frac{\sum_{k=1}^\infty\mathbb{P}(M \ge k) \mathbb{P}(\tau_k>yx)}{\sum_{k=1}^\infty\mathbb{P}(M \ge k) \mathbb{P}(\tau_k>x)} \\
                &= \lim_{x \to \infty} \frac{\mathbb{P}(\tau_1>yx)\left(1+\sum_{k=2}^\infty\mathbb{P}(M \ge k) \frac{\mathbb{P}(\tau_k>yx)}{\mathbb{P}(\tau_1>yx)}\right)}{\mathbb{P}(\tau_1>x)\left(1+\sum_{k=2}^\infty\mathbb{P}(M \ge k) \frac{\mathbb{P}(\tau_k>x)}{\mathbb{P}(\tau_1>x)}\right)} \notag
            \end{align}
            we conclude that $Z \sim \text{RV}(\alpha)$.
        \end{proof}

        \begin{proof}[\textbf{Proof of Theorem \ref{babel}}]

            Consider a random variable $W_Z$ controlling the tails of the $Z_i$'s, let $H = \{M_k = m_k, M_r =m_r, M_{k-1}=m_{k-1}, M_{r-1}=m_{r-1}, m_{r-1}<m_r<m_{k-1}<m_k\}$, and note that
            \begin{align}\label{lavagna}
                & \frac{\mathbb{P}(Z_k>z,Z_r>z)}{\mathbb{P}(W_Z>z)} \\
                &=
                \frac{\mathbb{P}(\sum_{h=M_{k-1}+1}^{M_k} \tau_h > z, \sum_{h=M_{r-1}+1}^{M_r} \tau_h > z)}{\mathbb{P}(W_Z>z)} \notag \\
                & = \frac{1}{\mathbb{P}(W_Z>z)}
                \sum_{m_k,m_r,m_{k-1},m_{r-1}}\mathbb{P}\biggl( \sum_{h=M_{k-1}+1}^{M_k} \tau_h > z, \sum_{h=M_{r-1}+1}^{M_r} \tau_h > z \biggr| H\biggr) \mathbb{P}(H)\notag
            \end{align}
            Now, the above series is clearly uniformly convergent as the generic term of the series is bounded above by $\mathbb{P}(H)$ and $\sum_{m_k,m_r,m_{k-1},m_{r-1}}\mathbb{P}(H)=1$. This allows to take the limit term by term in \eqref{lavagna}. Hence, in order to prove asymptotic independence it is enough to check that (recall also the independence of the $M$'s with the $\tau$'s)
            \begin{align}\label{candela}
                \lim_{z \to \infty} \frac{\mathbb{P}(\sum_{h=m_{k-1}+1}^{m_k} \tau_h > z, \sum_{h=m_{r-1}+1}^{m_r} \tau_h > z)}{\mathbb{P}(W_Z>z)} =0.
            \end{align}
            Let us first shorten the notation by letting $\mathcal{H} = \{ m_{k-1}+1, \dots, m_k \}$ and $\mathcal{F} = \{ m_{r-1}+1, \dots, m_r \}$. In order to prove \eqref{candela} consider the numerator:
            \begin{align}
                \mathbb{P} &\Bigl(\sum_{h \in \mathcal{H}} \tau_h > z, \sum_{h \in \mathcal{F}} \tau_h > z\Bigr) \\
                = {} & \sum_{j,l} \mathbb{P}\Bigl(\sum_{h \in \mathcal{H}} \tau_h > z, \sum_{h \in \mathcal{F}} \tau_h > z \bigr| \tau_j \ge \max_{h \in \mathcal{H}, h\ne j} \tau_h, \tau_l \ge \max_{h \in \mathcal{F},h\ne l} \tau_h\Bigr) \notag \\
                & \times \mathbb{P}\Bigl(\tau_j \ge \max_{h \in \mathcal{H}, h\ne j} \tau_h, \tau_l \ge \max_{h \in \mathcal{F},h\ne l} \tau_h\Bigr) \notag\\
                \le {} & \sum_{j,l} \mathbb{P} \Bigl( (m_k-m_{k-1})\tau_j>z, (m_r-m_{r-1})\tau_l>z \bigr| \tau_j \ge \max_{h \in \mathcal{H}, h\ne j} \tau_h, \tau_l \ge \max_{h \in \mathcal{F},h\ne l} \tau_h \Bigr) \notag\\
                & \times \mathbb{P}\Bigl(\tau_j \ge \max_{h \in \mathcal{H}, h\ne j} \tau_h, \tau_l \ge \max_{h \in \mathcal{F},h\ne l} \tau_h\Bigr) \notag \\
                = {} & \sum_{j,l}  \mathbb{P} \Bigl( (m_k-m_{k-1})\tau_j>z, (m_r-m_{r-1})\tau_l>z , \tau_j \ge \max_{h \in \mathcal{H}, h\ne j} \tau_h, \tau_l \ge \max_{h \in \mathcal{F},h\ne l} \tau_h \Bigr) \notag \\
                \le {} & \sum_{j,l} \mathbb{P} \Bigl( (m_k-m_{k-1})\tau_j>z, (m_r-m_{r-1})\tau_l>z \Bigr). \notag \\
                \le {} & \sum_{j,l} \mathbb{P}(\tau_j > z/y, \tau_l > z/y) \notag
            \end{align}
            where in the last step $y = \max(m_k-m_{k-1}, m_r-m_{r-1})$.
            Hence,
            \begin{align}
                & \frac{\mathbb{P}(\sum_{h \in \mathcal{H}} \tau_h > z, \sum_{h \in \mathcal{F}} \tau_h > z)}{\mathbb{P}(W_Z>z)}
                \le \sum_{j,l} \frac{\mathbb{P}(\tau_j > z/y, \tau_l > z/y)}{\mathbb{P}(W_Z>z)} \\
                & = \sum_{j,l} \frac{\mathbb{P}(\tau_j > z/y, \tau_l > z/y)}{\mathbb{P}(W_\tau>z/y)} \frac{\mathbb{P}(W_\tau> z/y)}{\mathbb{P}(W_Z>z)}, \notag
            \end{align}
            where $W_\tau$ is a random variable controlling the tails of the $\tau_j$'s. Notice that the first factor of each summand goes to zero with $z\to \infty$ for the asymptotic independence of the $\tau_j$'s. Let us consider the second factor:
            \begin{align}\label{prpro}
                \frac{\mathbb{P}(W_\tau> z/y)}{\mathbb{P}(W_Z>z)}
                = \frac{\frac{\mathbb{P}(W_\tau> z/y)}{\mathbb{P}(\tau_h>z/y)}}{\frac{\mathbb{P}(W_Z>z)}{\mathbb{P}(Z_j>z)}} \frac{\mathbb{P}(\tau_h>z/y)}{\mathbb{P}(Z_j>z)},
            \end{align}
            where $Z_j$ is chosen such that the ratio
            \begin{align}
                \frac{\mathbb{P}(Z_j > z)}{\mathbb{P}(W_Z > z)}
            \end{align}
            converges to a strictly positive finite constant, as $z$ goes to infinity. Moreover,
            $\tau_h$ is chosen to be one of the addends of $Z_j$ (recall that the $\tau$'s are asymptotically equivalent).
            Note that the numerator and the denominator of the first ratio of the above formula \eqref{prpro}, as $z \to \infty$ are both positive and bounded. The last factor is positive and asymptotically bounded as well. Since the number of terms of $Z_j$ is given by $M_j - M_{j-1}$, an application of 
            Theorem 2 of \cite{MR2745413}  (see also Theorem \eqref{vavava}),  ensures that $\mathbb{P}(Z_j > z) \sim \sum_{k=1}^\infty \mathbb{P}(M_j - M_{j-1} \ge k) \mathbb{P}(\tau_k>z)$. Note that the latter series is uniformly convergent for the finiteness of the (if necessary, conditional) expectation of $M_j - M_{j-1}$ (recall $p \ne 1/2$). Hence, 
            \begin{align}
                \lim_{z \to \infty} \frac{\mathbb{P}(Z_j > z)}{\mathbb{P}(\tau_h>z/y)}
                & = \lim_{z \to \infty} \sum_{k=1}^\infty \mathbb{P}(M_j - M_{j-1} \ge k) \frac{\mathbb{P}(\tau_k>z)}{\mathbb{P}(\tau_h>z/y)} \\
                & \le \lim_{z \to \infty} \sum_{k=1}^\infty \mathbb{P}(M_j - M_{j-1} \ge k) \frac{\mathbb{P}(\tau_k>z)}{\mathbb{P}(\tau_h>z)} \notag \\
                & = c_{k,h} \mathbb{E} \big(M_j - M_{j-1}\big), \notag
            \end{align}
            where $c_{k,h} = \lim_{z \to \infty} \mathbb{P}(\tau_k>z)/\mathbb{P}(\tau_h>z)$, is finite. This proves \eqref{candela}.
        \end{proof}

        \begin{proof}[\textbf{Proof of Theorem \ref{tazzina}}]
            Consider the vector $\bm{T_1}$ and denote by $g$ the measurable transformation $g \colon \mathbb{R}^n \to \mathbb{R}^2$ such that $g(\bm{T_1}) = (Z_A,Z_B)$ representing the IF mechanism. Notice that we have
            $g(a\bm{T_1}) = ag(\bm{T_1})$ for every positive constant $a\in \mathbb{R}_+$. Hence,

            \begin{align}
                \lim_{x \to \infty} \frac{\mathbb{P}(x^{-1}(Z_A,Z_B) \in [\bm{0},\bm{t}]^c)}{\mathbb{P}(x^{-1}(Z_A,Z_B) \in [\bm{0},\bm{1}]^c)}
                & = \lim_{x \to \infty} \frac{\mathbb{P}(x^{-1}g(\bm{T_1}) \in [\bm{0},\bm{t}]^c)}{\mathbb{P}(x^{-1}g(\bm{T_1}) \in [\bm{0},\bm{1}]^c)} \\
                & = \lim_{x \to \infty} \frac{\mathbb{P}(g(x^{-1}\bm{T_1}) \in [\bm{0},\bm{t}]^c)}{\mathbb{P}(g(x^{-1}\bm{T_1}) \in [\bm{0},\bm{1}]^c)} \notag \\
                & = \lim_{x \to \infty} \frac{\mathbb{P}(x^{-1}\bm{T_1} \in g^{-1} ([\bm{0},\bm{t}]^c))}{\mathbb{P}(x^{-1}\bm{T_1} \in g^{-1} ([\bm{0},\bm{1}]^c))} \notag \\
                & = \bar{\mu} ([\bm{0},\bm{t}]^c),\notag
            \end{align}
            where $g^{-1}(B)$ is the inverse image of the Borel set $B$ of $\mathbb{R}^2$ through $g$ and $\bar{\mu} = \mu \circ g^{-1}$.
            This proves \textbf{H1} as: a) the marginal regular variation has been already proven (see Theorem \ref{vavava}); b) the above reasoning can be extended to any vector $(Z_{A,j},Z_{B,k})$, with fixed $j,k \in \mathbb{N}^*$, of output ISIs including a common event (and thus being superimposed in the time evolution of the output ISIs processes).
        \end{proof}

        \begin{proof}[\textbf{Proof of Theorem \ref{ark}}]

            To prove $\textbf{H4}$ we consider the dependent infinite sequences $(Z_{A,j})_{j \in \mathbb{N}^*}$, $(Z_{B,k})_{k \in \mathbb{N}^*}$, of ISIs relative respectively to the receiving neurons $v_A$ and $v_B$. Since their components are asymptotically equivalent, we consider a positive random variable $W_Z$ controlling the tails of the components of both sequences. In other words, for every $j,k \in \mathbb{N}^*$,
            \begin{align}
                \lim_{z \to \infty}\frac{\mathbb{P}(Z_{A,j}>z)}{\mathbb{P}(W_Z>z)} = c_{A,j} \in (0,\infty), \qquad
                \lim_{z \to \infty}\frac{\mathbb{P}(Z_{B,k}>z)}{\mathbb{P}(W_Z>z)} = c_{B,k} \in (0,\infty).
            \end{align}
            We now fix $j$ and $k$ so that $Z_{A,j}$ and $Z_{B,k}$ are temporally superimposed, i.e.\ they contain the same firing event, and check the behaviour of the joint survival function compared with that of the marginals:
                \begin{align}
                \frac{\mathbb{P}(Z_{A,j}>z,Z_{B,k}>z)}{\mathbb{P}(W_Z>z)} = \frac{\mathbb{P}(\sum_{i=M_{j-1}+1}^{M_j} \tau_{A,i}>z,\sum_{i=N_{k-1}+1}^{N_k} \tau_{B,i}>z)}{\mathbb{P}(W_Z>z)}
            \end{align}
            Denote the event $F = \{M_j = m_j, N_k =n_k, M_{j-1}=m_{j-1}, N_{k-1}=n_{k-1}, n_{k-1}<n_r, m_{j-1}<m_j\}$. Recalling that $(\tau_{A,i})_{i \in \mathbb{N}^*}$, $(\tau_{B,i})_{i \in \mathbb{N}^*}$, $M_j$, $M_{j-1}$, $N_k$, $N_{k-1}$, are all independent, we have,
            \begin{align}
                &\frac{\mathbb{P}(Z_{A,j}>z,Z_{B,k}>z)}{\mathbb{P}(W_Z>z)} \\
                &= \frac{1}{\mathbb{P}(W_Z>z)} \sum_{m_j, m_{j-1}, n_k, n_{k-1}} \mathbb{P}\Bigl(\sum_{i=m_{j-1}+1}^{m_j} \tau_{A,i}>z,\sum_{i=n_{k-1}+1}^{n_k} \tau_{B,i}>z\Bigr) \mathbb{P}(F). \notag
            \end{align}
            Since the latter sum is bounded above by $\sum_{m_j, m_{j-1}, n_k, n_{k-1}} \mathbb{P}(F) = 1$, it is enough to consider the behaviour of
            \begin{align}
                \lim_{z \to \infty}\frac{\mathbb{P}\Bigl(\sum_{i=m_{j-1}+1}^{m_j} \tau_{A,i}>z,\sum_{i=n_{k-1}+1}^{n_k} \tau_{B,i}>z\Bigr)}{\mathbb{P}(W_Z>z)}.
            \end{align}
            Recall the two ISIs $Z_{A,j}$ and $Z_{B,k}$ are temporally superimposed and hence we can select $\tau_{A,*}$ and $\tau_{B,\circ}$, still temporally superimposed such that
            \begin{align}\label{tomo}
                \lim_{z \to \infty}& \frac{\mathbb{P}\Bigl(\sum_{i=m_{j-1}+1}^{m_j} \tau_{A,i}>z,\sum_{i=n_{k-1}+1}^{n_k} \tau_{B,i}>z\Bigr)}{\mathbb{P}(W_Z>z)} \ge \lim_{z \to \infty} \frac{\mathbb{P}\Bigl( \tau_{A,*}>z, \tau_{B,\circ}>z\Bigr)}{\mathbb{P}(W_Z>z)}  \\
                & = \lim_{z \to \infty} \frac{\mathbb{P}\Bigl( \min_{i \in \{ 1,\dots, \overline{n} \}} \Theta_{A,*}^i > z,  \min_{i \in \{  \underline{n},\dots, n \}} \Theta_{B,\circ}^i > z \Bigr)}{\mathbb{P}(W_Z>z)} \notag \\
                & = \lim_{z \to \infty} \frac{\mathbb{P}\Bigl(  \Theta_{A,*}^1 > z, \dots, \Theta_{A,*}^{\overline{n}} > z, \Theta_{B,\circ}^{\underline{n}} > z, \dots, \Theta_{B,\circ}^n>z  \Bigr)}{\mathbb{P}(W_Z>z)}.
                \notag
            \end{align}
            Notice that, due to the temporal superposition of the two considered output ISIs, the vectors $\bm{\Theta_{A,*}}$ and $\bm{\Theta_{B,\circ}}$ are subset of the same vector of inter-arrival times. Denote it by $\mathcal{T}$ and note that the components of $\mathcal{T}$, of $\bm{\Theta_{A,*}}$, and of $\bm{\Theta_{B,\circ}}$ are asymptotically equivalent.
            Hence,
            \begin{align}\label{barbis}
                \lim_{z \to \infty}& \frac{\mathbb{P}\Bigl(\sum_{i=m_{j-1}+1}^{m_j} \tau_{A,i}>z,\sum_{i=n_{k-1}+1}^{n_k} \tau_{B,i}>z\Bigr)}{\mathbb{P}(W_Z>z)}
                \ge 
                \lim_{z \to \infty} \frac{\mathbb{P}\Bigl(  \mathcal{T}^1 > z, \dots,  \mathcal{T}^n>z  \Bigr)}{\mathbb{P}(W_Z>z)}.
            \end{align}
            Since the components of $\mathcal{T}$ are taken to be fully dependent (hypothesis \textbf{H4} on input ISIs), the limit in the last line of \eqref{barbis} equals a strictly positive constant. To see that this is true, let $W_\Theta$ be the random variable controlling the tail of the components of the vector of forward recurrence time intervals containing $\bm{\Theta_{A,*}}$ and $\bm{\Theta_{B,\circ}}$. Note that $W_\Theta$ controls also the tail of the components of $\mathcal{T}$. Further,
            \begin{align}
                \lim_{z \to \infty} \frac{\mathbb{P}(W_\Theta>z)}{\mathbb{P}(W_Z>z)}
                = \lim_{z \to \infty} \frac{\mathbb{P}(W_\Theta>z)}{\mathbb{P}(Z_{A,j}>z)}
            \end{align}
            by the definition of $W_Z$. Then, consider that
            \begin{align}
                \frac{\mathbb{P}(W_\Theta>z)}{\mathbb{P}\Bigl(\sum_{i=M_{j-1}+1}^{M_j} \tau_{A,i}>z\Bigr)}
                &= \frac{\mathbb{P}(W_\Theta>z)}{\sum_{m_j,m_{j-1}}\mathbb{P}\Bigl(\sum_{i=M_{j-1}+1}^{M_j} \tau_{A,i}>z| \mathcal{M}\Bigr)\mathbb{P}(\mathcal{M)}} \\
                & = \frac{\mathbb{P}(W_\Theta>z)}{\sum_{m_j,m_{j-1}}\mathbb{P}\Bigl(\sum_{i=m_{j-1}+1}^{m_j} \tau_{A,i}>z\Bigr)\mathbb{P}(\mathcal{M)}} \notag
            \end{align}
            where $\mathcal{M}= \{ M_j=m_j,M_{j-1}= m_{j-1} \}$ and noting that in the last step we used the independence between the $\tau$'s and the $M$'s.
            Hence, it is sufficient to prove for every $(m_j,m_{j-1}) \in \text{supp}\bigl((M_j,M_{j-1})\bigr)$ that
            \begin{align}
                \frac{\mathbb{P}\Bigl(\sum_{i=m_{j-1}+1}^{m_j} \tau_{A,i}>z\Bigr)}{\mathbb{P}(W_\Theta>z)} \longrightarrow c \in (0, \infty)
            \end{align}
            as $z \to \infty$. Let $B_j = \{m_{j-1}+1, \dots,m_j\}$. On the one hand, we have
            \begin{align}
                & \frac{\mathbb{P}\Bigl(\sum_{i=m_{j-1}+1}^{m_j} \tau_{A,i}>z\Bigr)}{\mathbb{P}(W_\Theta>z)} \\
                & = \frac{\sum_{\ell \in B_j} \mathbb{P}\Bigl(\sum_{i=m_{j-1}+1}^{m_j} \tau_{A,i}>z| \tau_{A,\ell}\ge\max_{i \in B_j, i \ne \ell} \tau_{A,i}\Bigr) \mathbb{P}( \tau_{A,\ell}\ge\max_{i \in B_j, i \ne \ell} \tau_{A,i})}{\mathbb{P}(W_\Theta>z)} \notag \\
                & \le \frac{\sum_{\ell \in B_j}\mathbb{P} \Bigl((m_j-m_{j-1}) \tau_{A,\ell}>z \Bigr)}{\mathbb{P}(W_\Theta>z)} \notag \\
                & = \frac{\sum_{\ell \in B_j}\mathbb{P} \Bigl( \Theta_{A,\ell}^1> \frac{z}{(m_j-m_{j-1})}, \dots, \Theta_{A,\ell}^n> \frac{z}{(m_j-m_{j-1})} \Bigr)}{\mathbb{P}(W_\Theta>z)} \notag \\
                & \le \sum_{\ell \in B_j}\frac{\mathbb{P} \Bigl( \Theta_{A,\ell}^1> \frac{z}{(m_j-m_{j-1})} \Bigr)}{\mathbb{P}(W_\Theta>z)}. \notag 
            \end{align}
            Taking the limit for $z \to \infty$ we obtain
            \begin{align}
                \lim_{z \to \infty} \frac{\mathbb{P}\Bigl(\sum_{i=m_{j-1}+1}^{m_j} \tau_{A,i}>z\Bigr)}{\mathbb{P}(W_\Theta>z)} \le (m_j-m_{j-1})^{\alpha +1}.
            \end{align}
            On the other hand,
            \begin{align}
                \frac{\mathbb{P}\Bigl(\sum_{i=m_{j-1}+1}^{m_j} \tau_{A,i}>z\Bigr)}{\mathbb{P}(W_\Theta>z)}
                & \ge \frac{\mathbb{P}\Bigl( \tau_{A,m_j}>z\Bigr)}{\mathbb{P}(W_\Theta>z)} \\
                & = \frac{\mathbb{P}(\Theta_{A,m_j}^1>z, \dots, \Theta_{A,m_j}^n>z)}{\mathbb{P}(W_\Theta>z)} \notag
            \end{align}
            Again, by taking the limit for $z \to \infty$ and considering $\textbf{H4}$ on the input ISIs, we arrive at
            \begin{align}
                \lim_{z \to \infty} \frac{\mathbb{P}\Bigl(\sum_{i=m_{j-1}+1}^{m_j} \tau_{A,i}>z\Bigr)}{\mathbb{P}(W_\Theta>z)} \ge \text{const} \in (0,\infty).
            \end{align}
            This is enough to proving that \textbf{H4} holds for the output ISI vector $(Z_{A,j}, Z_{B,k})$.
        \end{proof}

	\end{appendices}

\bibliographystyle{abbrvnat}
\bibliography{neuron.bib}

\end{document}